\documentclass[12pt]{amsart}

\usepackage{calrsfs}
\usepackage{hhline}
\usepackage{amsfonts}
\usepackage{amsthm}
\usepackage[psamsfonts]{amssymb}
\usepackage[dvipdfmx]{graphicx}
\usepackage{graphicx}
\usepackage{ascmac}
\usepackage{amsmath}
\usepackage{fancybox}
\usepackage{fancyhdr}
\usepackage{enumerate}
\usepackage{mathrsfs}
\usepackage{color}
\usepackage{multicol}
\usepackage{framed}
\usepackage{accents}

\newtheorem{thm}{Theorem}[section]
\newtheorem{prop}[thm]{Proposition}

\newtheorem{rem}[thm]{\it Remark}

\makeatletter
\def\widebar{\accentset{{\cc@style \mskip10mu}}}
\def\wideubar{\underaccent{{\cc@style \mskip10mu}}}
\makeatother

\newtheorem*{Acknowledgements}{Acknowledgements}

\address{Department of Mathematics, Graduate School of Science, Osaka University, Toyonaka, Osaka 560-0043, Japan}
\email{h-tadano@cr.math.sci.osaka-u.ac.jp}
\title{A note on lower bounds for the first eigenvalue \\ of the Witten-Laplacian}
\author{Homare TADANO}
\date{May 3, 2014.}

\begin{document}

\begin{abstract}
In this note, by extending the arguments of Ling (Illinois J. Math. 51, 853-860, 2007) to Bakry-\'{E}mery geometry, we shall give lower bounds for the first nonzero eigenvalue of the Witten-Laplacian on compact Bakry-\'{E}mery manifolds in the case that the Bakry-\'{E}mery Ricci curvature has some negative lower bounds and the manifold has the symmetry that the minimum of the first eigenfunction is the negative of the maximum. Our estimate is optimal among those obtained by a self-contained method.
\end{abstract}

\maketitle

\numberwithin{equation}{section}

\section{Introduction}

A \textit{Bakry-\'{E}mery manifold} $(M, g, f)$ is a Riemannian manifold $(M, g)$ equipped with the weighted volume form $d \mu := e^{- f} d \mathrm{vol}_{g}$, where $f \in C^{2}(M)$ is a real-valued $C^{2}$-function on $M$. A Bakry-\'{E}mery manifold was introduced by Bakry-\'{E}mery \cite{Bakry-Emery} and has been received much attention in various areas of mathematics. Given a Bakry-\'{E}mery manifold $(M, g, f)$, we obtain a \textit{Bakry-\'{E}mery Ricci curvature} and a \textit{Witten-Laplacian} defined by
\begin{equation}\label{Ric-Delta}
\operatorname{Ric}_{f} := \operatorname{Ric} + \operatorname{Hess} f, \quad \Delta_{f} := \Delta - \nabla f \cdot \nabla, 
\end{equation}
respectively. Here, $\Delta = g^{ij} \nabla_{i} \nabla_{j}$. As usual, for any functions $u, v \in C_{0}^{\infty}(M)$ on $M$ with a compact support, the following integration by parts formula holds:
\[
\int_{M} \left< \nabla u, \nabla v \right> d \mu = - \int_{M} (\Delta_{f} u) v d \mu = - \int_{M} u (\Delta_{f} v) d \mu, \quad u, v \in C_{0}^{\infty}(M).
\]
Moreover, Bakry and \'{E}mery \cite{Bakry-Emery} proved that
\[
\frac{1}{2} \Delta_{f} | \nabla u |^{2} = | \operatorname{Hess} u |^{2} + \left< \nabla u, \nabla \Delta_{f} u \right> + \operatorname{Ric}_{f}(\nabla u, \nabla u), \quad u \in C_{0}^{\infty}(M).
\]
This formula can be considered as a natural extention of the Bochner-Weitzenb\"{o}ck formula. The Bakry-\'{E}mery Ricci curvature and the Witten-Laplacian are good substitutes of the Ricci curvature and the Laplacian respectively extending many fundamental theorems in Riemannian geometry to Bakry-\'{E}mery geometry, for example, Myers type theorems \cite{Qian, Lott, Wei-Wylie, Limoncu, Zhang}, Cheeger-Gromoll splitting theorems \cite{Fang-Li-Zhang}, Gradient estimates \cite{Bakry-Qian1, Wu1, Zhao}, eigenvalue estimates \cite{Bakry-Qian2, Futaki-Sano, Andrews-Ni, Futaki-Li-Li, Wu2}, Liouville type theorems \cite{Li1, Naber, Ruan, Huang-Zhang-Zhang, Brighton}. Moreover, it has been an important tool in Optimal transport theory \cite{Villani} and in Perelman's entropy formula for heat equations on complete Riemannian manifolds \cite{Li2, Li3}.

The aim of the present paper is to give lower bounds for the first nonzero eigenvalue $\lambda_{1}$ of the Witten-Laplacian $\Delta_{f}$ on compact Bakry-\'{E}mery manifolds in the case of the Bakry-\'{E}mery Ricci curvature is bounded from below by some negative constants and the manifold has the symmetry that the minimum of the first eigenfunction is the negative of the maximum. Our main result is the following:

\begin{thm}\label{Main-Theorem}
Let $(M, g, f)$ be an $n$-dimensional compact Bakry-\'{E}mery manifold. Suppose that the Bakry-\'{E}mery Ricci curvature has the lower bound
\begin{equation}\label{assumption}
\operatorname{Ric}_{f} \geqslant K
\end{equation}
for some negative constants $K < 0$ and the manifold has the symmetry that the minimum of the first eigenfunction of the Witten-Laplacian $\Delta_{f}$ is the negative of the maximum. Then the first nonzero eigenvalue $\lambda_{1}$ of the Witten-Laplacian has the lower bound
\begin{equation}\label{Main-Theorem-eq}
\lambda_{1} \geqslant \frac{\pi^{2}}{d^{2}} + \frac{1}{2}K, 
\end{equation}
where $d$ denotes the diameter of $(M, g)$.
\end{thm}

\begin{rem}\rm
If $f$ is constant in (\ref{Ric-Delta}), the Bakry-\'{E}mery Ricci curvature and the Witten-Laplacian become the ordinary Ricci curvature $\operatorname{Ric}$ and Laplacian $\Delta$, respectively. The study of the first nonzero eigenvalue for the Laplacian has a long history. See \cite{Chavel, Shi-Zhang, Lichnerowicz, Cheng, Li-Yau, Zhong-Yang, Yang, Zhao-Yang, Ling-Illinois, Ling-AGAG, Andrews-Clutterbuck} and references therein. Our estimate (\ref{Main-Theorem-eq}) can be considered as a extension of Ling \cite{Ling-Illinois} to the Witten-Laplacian via Bakry-\'{E}mery Ricci curvature.
\end{rem}


By combining Proposition 3.8 in \cite{Futaki-Sano} and above Theorem \ref{Main-Theorem}, we immediately obtain:

\begin{thm}\label{theorem}
Let $(M, g, f)$ be an $n$-dimensional compact Bakry-\'{E}mery manifold. Suppose that the Bakry-\'{E}mery Ricci curvature has the lower bound {\rm (\ref{assumption})} for some constants $K \in \mathbb{R}$ and the manifold has the symmetry that the minimum of the first eigenfunction of the Witten-Laplacian $\Delta_{f}$ is the negative of the maximum. Then the first nonzero eigenvalue $\lambda_{1}$ of the Witten-Laplacian has the lower bound {\rm (\ref{Main-Theorem-eq})}.
\end{thm}

\begin{rem}\rm
Under only the assumption (\ref{assumption}) for some positive constants $K > 0$, by using the method of modulus of continuity, Andrews and Ni \cite{Andrews-Ni} showed that the first nonzero eigenvalue $\lambda_{1}$ of the Witten-Laplacian $\Delta_{f}$ on compact Bakry-\'{E}mery manifolds satisfies (\ref{Main-Theorem-eq}).
\end{rem}

\begin{rem}\rm
Under only the assumption (\ref{assumption}) for some constants $K \in \mathbb{R}$, Futaki \textit{et al} \cite{Futaki-Li-Li} showed that the first nonzero eigenvalue $\lambda_{1}$ of the Witten-Laplacian $\Delta_{f}$ on compact Bakry-\'{E}mery manifolds satisfies
\begin{equation}\label{Futaki-Li-Li-eq}
\lambda_{1} \geqslant \sup_{s \in (0, 1)} \left\{ 4s(1 - s) \frac{\pi^{2}}{d^{2}} + sK \right\}  = 
\begin{cases}
0 & \textrm{if} \quad Kd^{2} < - 4 \pi^{2}, \\
(\frac{\pi}{d} + \frac{Kd}{4 \pi})^{2} & \textrm{if} \quad Kd^{2}\in [- 4 \pi^{2}, 4 \pi^{2}], \\
K & \textrm{if} \quad Kd^{2} \in (4 \pi^{2}, (n - 1) \pi^{2}].
\end{cases}
\end{equation}
At the present time, the above estimate (\ref{Futaki-Li-Li-eq}) is optimal. For $K < 0$, taking $s = \frac{1}{2}$ in the above, we recapture (\ref{Main-Theorem-eq}). Hence the lower bound (\ref{Futaki-Li-Li-eq}) is stronger than that in Theorem \ref{theorem}. An advantage of our result is being self-contained as we see below, while they use a comparison theorem between eigenvalue problems. Note that our estimate in Theorem \ref{theorem} is optimal among those obtained by a self-contained method.
\end{rem}

\begin{Acknowledgements}\rm
The author would like to thank his supervisor Professor Toshiki Mabuchi. The author also wish to thank Professor Akito Futaki for his comments. This work was partly supported by Moriyasu graduate student fellowship.
\end{Acknowledgements}

\section{Proof of Theorem \ref{Main-Theorem}}

In this section, following \cite{Futaki-Sano} we extend the arguments in \cite{Ling-Illinois} to Bakry-\'{E}mery geometry in the case that the minimum of the first eigenfunction of the Witten-Laplacian $\Delta_{f}$ is the negative of the maximum. Let $u$ be an eigenfunction of the first nonzero eigenvalue $\lambda_{1}$ for the Witten-Laplacian, i.e., $\Delta_{f} u + \lambda_{1} u = 0$. We may normalize the function $u$ such that
\begin{equation}\label{max-min}
\max_{M} u = 1, \quad \min_{M} u = - 1.
\end{equation}
Let $b > 1$ be an arbitrary constant. Define a function $Z$ on $[- \sin^{- 1}(1 / b), \sin^{- 1}(1 / b)]$ and  a constant $\delta$ by
\begin{equation}\label{Z-delta}
Z(t) := \max_{x \in U(t)} \frac{| \nabla u | ^{2}}{\lambda_{1}(b^{2} - u^{2})}, \quad \delta := \frac{K}{2 \lambda_{1}} < 0, 
\end{equation}
where $U(t) := \{ x \in M : \sin^{- 1}(u(x) / b) = t \}$. Note that $t \in [- \sin^{- 1}(1 / b), \sin^{- 1}(1 / b) ]$.

\begin{rem}\label{rem}\rm
In Proposition 3.1 and 3.4 of \cite{Futaki-Sano}, estimates for an eigenvalue $\lambda$ and a function $Z(t)$ are obtained, respectively. On the other hand, the same way does not work in our case, since $K < 0$. However, the same argument as Proposition 3.5 in \cite{Futaki-Sano} still holds in our case. Note that (\ref{max-min}) corresponds to the case of $a = c = 0$ in \cite{Futaki-Sano}.
\end{rem}

As Proposition 3.6 (b) in \cite{Futaki-Sano}, we have:

\begin{prop}\label{prop}
Suppose that the function $z : [ - \sin^{- 1}(1 / b), \sin^{- 1}(1 / b)] \rightarrow \mathbb{R}$ satisfies the following conditions:

\begin{enumerate}[{\rm (1)}]

\item $z(t) \geqslant Z(t)$ for all $t \in [ - \sin^{- 1}(1 / b), \sin^{- 1}(1 / b)]$, 

\item there exists some $x_{0} \in M$ such that $z(t_{0}) = Z(t_{0})$ at $t_{0} = \sin^{- 1}(u(x_{0}) / b)$, 

\item $z(t_{0}) \geqslant 1$, and

\item $\dot{z}(t_{0}) \sin t_{0} \leqslant 0$.

\end{enumerate}

Then we have the following:
\begin{equation}\label{ODE-prop}
z(t_{0}) \leqslant \frac{1}{2} \ddot{z}(t_{0}) \cos^{2} t_{0} - \dot{z}(t_{0}) \cos t_{0} \sin t_{0} + 1 - 2 \delta \cos^{2} t_{0}.
\end{equation}
\end{prop}

\begin{proof}[Sketch of the Proof]
By the same arguments as Proposition 3.5 in \cite{Futaki-Sano}, we have
\[
\begin{aligned}
0 & \leqslant \frac{1}{2} \ddot{z}(t_{0}) \cos^{2} t_{0} - \dot{z}(t_{0}) \cos t_{0} \sin t_{0} - z(t_{0}) + 1 - 2 \delta \cos^{2} t_{0} \\
& \quad - \frac{\dot{z}(t_{0})}{4 z(t_{0})} \cos t_{0} \{ \dot{z}(t_{0}) \cos t_{0} - 2 z(t_{0}) \sin t_{0} + 2 \sin t_{0} \}.
\end{aligned}
\]
By the assumption (3) and (4), the last term of the above is nonpositive.
\end{proof}

Now, we give a proof of Theorem \ref{Main-Theorem}.

\begin{proof}[Proof of Theorem {\rm \ref{Main-Theorem}}]
The proof is the same as Theorem 3.1 in \cite{Ling-Illinois}. Then, we just give an outline of the proof. Let
\[
z(t) := 1 + \delta \xi(t), 
\]
where $\xi(t)$ is a function on $[- \pi / 2, \pi / 2]$ defined by
\[
\xi(t) = \frac{\cos^{2} t + 2t \sin t \cos t + t^{2} - \frac{\pi^{2}}{4}}{\cos^{2} t}
\]
and $\delta < 0$ is the negative constant as in (\ref{Z-delta}).
This function $\xi(t)$ is introduced by Ling and needed properties are studied in \cite{Ling-Illinois} and \cite{Ling-AGAG}. By using such properties, we have
\[
\begin{aligned}
& \frac{1}{2} \ddot{z}(t) \cos^{2} t - \dot{z} \cos t \sin t - z = - 1 + 2 \delta \cos^{2} t, \\
& z(t) \geqslant 1, {\ \rm and} \\
& \dot{z}(t) \sin t \leqslant 0.
\end{aligned}
\]
By using the above three conditions and (\ref{ODE-prop}), we can show that
\[
z(t) \geqslant Z(t), \quad t \in [- \sin^{- 1}(1 / b), \sin^{-1 }(1 / b)].
\]
The above implies
\begin{equation}\label{eq}
\sqrt{\lambda_{1}} \geqslant \frac{| \nabla t |}{\sqrt{z(t)}}, \quad t \in [- \sin^{- 1}(1 / b), \sin^{- 1}(1 / b)].
\end{equation}
Let $q_{1}$ and $q_{2}$ be the two points such that $u(q_{1}) = 1$ and $u(q_{2}) = -1$, respectively. Let $L$ be the minimizing geodesic between $q_{1}$ and $q_{2}$. We integrate the both sides of (\ref{eq}) along $L$ and change variable. Letting $b \rightarrow 1$, we have
\[
\begin{aligned}
d \sqrt{\lambda_{1}} & \geqslant \int_{L} \frac{| \nabla t |}{\sqrt{z(t)}} dl = \int_{- \frac{\pi}{2}}^{\frac{\pi}{2}} \frac{1}{\sqrt{z(t)}} dt \geqslant \frac{\left( \int_{- \pi / 2}^{\pi / 2} dt \right)^{3 / 2}}{\left( \int_{- \pi / 2}^{\pi / 2} z(t) dt \right)^{1 / 2}} \geqslant \left( \frac{\pi^{3}}{\int_{- \pi / 2}^{\pi / 2} z(t) dt} \right)^{1 / 2}.
\end{aligned}
\]
By the definition of $z(t)$ and the properties of $\xi(t)$, we have
\[
\lambda_{1} \geqslant \frac{\pi^{3}}{d^{2} \int_{- \pi / 2}^{\pi / 2} z(t) dt} = \frac{\pi^{2}}{d^{2}(1 - \delta)} \quad {\rm and} \quad \lambda_{1} \geqslant \frac{\pi^{2}}{d^{2}} + \frac{1}{2}K.
\]
The proof of Theorem \ref{Main-Theorem} is completed.
\end{proof}

\begin{rem}\rm
In \cite{Futaki-Sano}, under only the assumption (\ref{assumption}) for some constants $K > 0$, it is proved that the first nonzero eigenvalue $\lambda_{1}$ of the Witten-Laplacian $\Delta_{f}$ on compact Bakry-\'{E}mery manifolds has the lower bound $\lambda_{1} \geqslant \frac{\pi^{2}}{d^{2}} + \frac{31}{100}K$. However, in the case of $K < 0$ same argument as in \cite{Futaki-Sano} does not hold, since there is no substitute for the Lichnerowicz type estimate. See the Case (B-a-2) in \cite{Futaki-Sano} and the Case (II-a) in \cite{Ling-AGAG}.
\end{rem}

\end{document}